\documentclass[12pt,reqno]{amsart}
\advance \topmargin by -\headheight
\advance \topmargin by -\headsep
\evensidemargin \oddsidemargin
\usepackage{amsmath}
\usepackage{amsfonts}
\usepackage{stmaryrd}
\usepackage{amssymb}
\usepackage{amsthm}
\usepackage{csquotes}
\usepackage{color}
\usepackage{mathrsfs}
\usepackage{dsfont}
\usepackage{bm}
\usepackage{cite}
\usepackage{soul}

\numberwithin{equation}{section}
\renewcommand\vec{\bm}
\newcommand{\n}[1]{\|{#1}\|}

\newtheorem{theorem}{Theorem}[section]

\newtheorem{lemma}[theorem]{Lemma}
\newtheorem{Proposition}[theorem]{Proposition}
\newtheorem{Conjecture}[theorem]{Conjecture}
\newtheorem{Corollary}[theorem]{Corollary}

\theoremstyle{definition}

\theoremstyle{remark}

\def\ZZ{\mathbb{Z}}
\def\RR{\mathbb{R}}
\def\NN{\mathbb{N}}

\def\FF{\mathbb{F}}

\def\TT{\mathbb{T}}
\def\supp{\rm supp}

\def\L{\mathcal{L}}

\def\A{\mathcal{A}}

\def\M{\mathrm{Mat}}

\usepackage{mathtools}

\title{On Commuting pairs in Arbitrary sets of $2\times 2$ matrices}

\author{Akshat Mudgal}
\address{Mathematics Institute, Zeeman Building, University of Warwick, Coventry CV4 7AL, UK}
\email{Akshat.Mudgal@warwick.ac.uk}


\renewcommand\vec{\bm}

\begin{document}

\def\TT{\mathbb{T}}
\def\RR{\mathbb R}
\def\d{\,\mathrm d}
\def\ZZ{\mathbb{Z}}
\def\cS{\mathcal{S}}
\def\cY{\mathcal{Y}}
\def\cB{\mathcal{B}}

\maketitle

\begin{abstract}
Let $\textrm{Mat}_2(\mathbb{R})$ be the set of $2 \times 2$ matrices with real entries. For any $\varepsilon>0$ and any finitely--supported probability measure $\mu$ on $\textrm{Mat}_2(\RR)$, we prove that either
\[ T(\mu) =  \sum_{X, Y \in {\rm supp}(\mu), XY = YX} \mu(X) \mu(Y) < \varepsilon \]
or there exists some finite set $\mathcal{S}$ contained in a $2$-dimensional subspace of $\textrm{Mat}_2(\mathbb{R})$ such that $\mu(\mathcal{S}) \geq \varepsilon/8$. This is sharp up to the multiplicative constant. We prove quantitatively stronger results when 
\[ \mu \bigg( \begin{pmatrix}
    a_1 & a_2\\
    a_3 & a_4
\end{pmatrix} \bigg) = \nu(a_1) \dots \nu(a_4) \ \  \text{for every} \ a_1, \dots, a_4 \in \RR, \]
 with $\nu$ being some finitely--supported probability measure on $\mathbb{R}$. For instance, when $\mathcal{A} \subset \mathbb{R}$ is a generalised arithmetic progression or multiplicative progression of dimension $d$ and $\nu = \mathds{1}_{\mathcal{A}}/|\A|$, our techniques imply that $|\A|^{-3} \ll_d T(\mu) \ll_d |\A|^{-3}$. Our methods highlight the connections of this problem to results in incidence geometry, growth in groups phenomenon as well as Bourgain--Chang type sum-product estimates over $\mathbb{R}$. 
 The latter includes applications of Schmidt's subspace theorem and the resolution of the weak polynomial Freiman--Ruzsa conjecture over integers. 
\end{abstract}

\section{Introduction}

Given a finite group $G$, the question of estimating the number $\mathfrak{C}(G)$ of \emph{commuting pairs} in $G\times G$, that is the number of pairs $(g,h) \in G \times G$ such that $gh=hg$, has had a rich history. Erd\H{o}s--Tur\'{a}n \cite{ET1968} seemed to have been the first to analyse this question for arbitrary finite groups as well as the cases when $G$ is the symmetric group $S_n$. Much work has since been done on this topic, see, for example, a note by Gustafson \cite{Gu1973} proving that $\mathfrak{C}(G) \leq 5|G|^2/8$ for all finite, non-abelian group $G$, or work of Neumann \cite{Ne1989} characterising all finite groups $G$ satisfying $\mathfrak{C}(G) \geq \delta |G|^2$ for any fixed $\delta >0$, or work of Eberhard \cite{Eb2015} proving that the set $\{ |G|^{-2} \mathfrak{C}(G)  :  G \ \text{a finite group}\}$ is well-ordered and all its limit points are rational.


These type of problems have also been analysed when $G$ is replaced by a suitable set of matrices. For example, given $d \in \mathbb{N}$, Feit--Fine \cite{FF1960} proved an exact formula for the number of commuting pairs of $d \times d$ matrices with elements arising from some finite field. More recently, for every $d \geq 3$, Browning--Sawin--Wang \cite{BSW2024} proved that the number of commuting pairs of $d \times d$ matrices with entries lying in $[-N,N]\cap \mathbb{Z}$ is $O_d(N^{d^2 + 2 - \frac{2}{d+1}})$. When $d=2$, the divisor function estimate implies that the number of such pairs is $O_{\varepsilon}(N^{5 + \varepsilon})$ for every $\varepsilon >0$.

There are two natural generalisations of this line of investigation, first, one might consider estimates for the number of commuting pairs $(A,B) \in \mathscr{A}^2$ where $\mathscr{A}$ is the set of all $d \times d$ matrices with entries arising from an arbitrary finite set $\mathcal{A} \subset \RR$. Secondly, given an arbitrary set $\mathscr{A}$ of $d \times d$ matrices with real entries, one might be interested in obtaining a structure theorem about when $\mathscr{A}$ has many commuting pairs.  In this paper, we investigate these two questions when $d=2$ and highlight their connections to topics in incidence geometry, sum-product phenomenon and growth in groups.


In order to set the scene, let $\textrm{Mat}_2(\mathbb{R})$ be the set of $2 \times 2$ matrices with real entries, let $\mu$ be a finitely-supported probability measure on $\M_2(\RR)$ and let ${\rm supp}(\mu)$ be the support of $\mu$. We are interested in studying 
\[ T(\mu) = \sum_{X, Y \in {\rm supp}(\mu)} \mu(X) \mu(Y) \mathds{1}_{XY = YX}. \]
Given some finite, non-empty set $\mathscr{A} \subset \M_2(\RR)$, denote $\dim(\mathscr{A})$ to be the dimension of the additive vector space generated by elements of $\mathscr{A}$ over $\RR$. This allows us to define 
\[ \delta(\mu) = \max_{\substack{ \mathcal{S'} \subseteq {\rm supp}(\mu) , \\  \dim(\mathcal{S}') \leq 2
  }} \mu(\mathcal{S}').\]
  With this notation in hand, we state our first result.

\begin{theorem} \label{t1}
    Let  $\mu$ be a finitely-supported probability measure on $\M_2(\RR)$. Then 
    \begin{equation} \label{ub1}
        T(\mu) \leq 8 \ \delta(\mu). 
    \end{equation} 
\end{theorem}


The conclusion of Theorem \ref{t1} is optimal up to the multiplicative constant. Indeed, for any $\varepsilon >0$, there exists a finitely-supported probability measure $\mu$ such that $T(\mu), \delta(\mu) < \varepsilon$ and $T(\mu) \geq (2/3 - \varepsilon)\delta(\mu)$, see \eqref{examplesharp} and \eqref{jrlk}.  Moreover, the definition of $\delta(\mu)$ is optimal, in the sense that one genuinely has to consider the maximum of $\mu(H \cap {\rm supp}(\mu))$ for infinitely many choices of $2$-dimensional subspaces $H$. Indeed, given non-zero $A, B \in \M_2(\RR)$ such that $A \neq \lambda B$ for any $\lambda\in \RR$ and $AB = BA$, let 
\[ \mathscr{A} = \{ i A + j B : i, j \in [N] \}  \  \ \text{and} 
\ \ \mu = N^{-2} \mathds{1}_{\mathscr{A}} , \]
where $[N]$ denotes the set $\{1,2,\dots, N\}$ and $\mathds{1}_{\mathscr{A}}$ is the indicator function of the set $\mathscr{A}$.
Since $XY = YX$ for every $X,Y \in \mathscr{A}$, we have
$T(\mu) = 1$.

One can strengthen Theorem \ref{t1} when $\mu$ has a product type structure. In particular, given a finitely-supported probability measure $\nu$ on $\RR$, we may induce a finitely-supported probability measure $\mu_{\nu}$ on $\M_2(\RR)$ by defining 
\[ \mu_{\nu}((a_{i,j})_{1 \leq i,j \leq 2}) = \nu(a_{1,1}) \dots \nu(a_{2,2}) \]
for every $a_{1,1}, \dots, a_{2,2} \in \RR$.
Next, we define the multiplicative energy $M(\nu)$ as 
\[ M(\nu) = \sum_{a_1, a_2, a_3, a_4 \in \supp(\nu)\setminus\{0\}} \nu(a_1) \nu(a_2) \nu(a_4) \nu(a_4) \mathds{1}_{a_1/a_2 = a_3/a_4} \]
and the norms 
\[ \n{\nu}_p = ( \sum_{x \in \supp(\nu)} \nu(x)^p )^{1/p} \ \ \text{and} \  \ \n{\nu}_{\infty} = \max_{x \in \supp(\nu)} \nu(x)\]
for every $1 \leq p < \infty$. We now state our second result.

\begin{theorem} \label{ptpl}
Let $\nu$ be a finitely-supported probability measure on $\RR$. Then 
\[ T(\mu_{\nu}) \ll  \n{\nu}_2^{4}  M(\nu)^{1/2}  +  \n{\nu}_{\infty}^2 M(\nu) + \n{\nu}_2^6 + \nu(0)^3.  \] 
\end{theorem}

In order to illustrate the above upper bound, we set $\nu = |\A|^{-1} \mathds{1}_{\A}$ for some finite, non-empty set $\A \subset \RR$. Given such an $\mathcal{A}$, we define
\[     T(\mathcal{A}) =  \sum_{A,B\in \mathscr{A}}\mathds{1}_{AB = BA}, \ \ \ \text{where} \ \ \ \mathscr{A} =  \bigg\{ \begin{pmatrix}
    a_1 & a_2\\
    a_3 & a_4
\end{pmatrix}  : a_i \in \mathcal{A}  \bigg\}. \]
Moreover, writing $M(\mathcal{A})$ to be the number of $a_1, a_2, a_3, a_4 \in \mathcal{A} \setminus \{0\}$ such that $a_1/ a_2 = a_3/ a_4$, Theorem \ref{ptpl} implies that whenever $|\A| \geq 2$, one has
\begin{equation} \label{nim}
   T(\mathcal{A})  \ll |\A|^{4} M (\mathcal{A})^{1/2} .
\end{equation} 
Inequality \eqref{nim} is sharp up to the implicit constant in Vinogradov notation when $0 \in \A$ and $M(\A) \ll |\A|^2$.

Before recording further upper bounds, we quickly consider two examples of sets $\A$ that provide many commuting pairs of matrices. Writing $I$ to be the $2 \times 2$ identity matrix, observe that for any $A \in {\rm Mat}_2(\RR)$ and any $\lambda \in \RR$, the matrices $A + \lambda I$ and $\lambda A$ always commute with $A$. Thus, setting $\A = [N]$, since we have $|\A \cap (\lambda + \A)| \gg N$ for at least $N/4$ many choices of $\lambda \in \mathbb{N}$, we see that $T(\A) \gg N^5$. Similarly, setting $\A = \{2,4,\dots, 2^N\}$, since $|\A \cap (\lambda \cdot \A)| \gg N$ for at least $N/4$ many choices of $\lambda \in \mathbb{N}$, we have $T(\A) \gg N^5$. 

Noting the above two examples, it is natural to consider estimates for $T(\A)$ when the underlying set $\A$ is either additively structured or multiplicatively structured. Thus, given a finite, non-empty set $\mathcal{A} \subset \RR$, we define the sumset $\A + \A$ and the additive doubling ${\rm K}$ as 
\[ \mathcal{A} + \mathcal{A} = \{a + b : a, b \in \mathcal{A}\} \ \ \text{and} \ \ {\rm K} = |\A+\A|/|\A| \]
and the product set $\A \cdot \A$ and the multiplicative doubling ${\rm M}$ as
\[ \A \cdot \A  = \{ a\cdot b : a, b \in \A\} \ \ \text{and} \ \ {\rm M} = |\A \cdot \A| /|\A| .\]
Roughly speaking, a set $\A$ may considered to be additively structured if ${\rm K}$ is small in terms of $\A$, say, for instance, when ${\rm K} \ll |\A|^{\varepsilon}$ for some small $\varepsilon>0$. A similar notion holds for multiplicative structure. As in the preceding two examples, one can show that if $\A$ is either additively or multiplicatively structured, then $T(\A)$ must be large. Moreover, our next two results will prove corresponding upper bounds of similar order in both of these cases.

Upon incorporating ideas from work of Solymosi \cite{So2009} on the sum-product conjecture along with some of the incidence geometric techniques involved in the proof of Theorem \ref{ptpl}, which were themselves inspired by work of Murphy--Roche-Newton--Shkredov \cite{MRNS2015}, we can prove almost optimal estimates for $T(\A)$ whenever ${\rm K}$ is small.

\begin{Corollary} \label{dmf2}
    Let ${\rm K} >1$ and let $\mathcal{A} \subset \RR$ be a finite, non-empty set with $|\A+\A| = {\rm K}|\A|$. Then
\[  {\rm K}^{-1} |\A|^5 \leq T(\mathcal{A}) \ll {\rm K}^{4/3} |\A|^{5} .  \]
\end{Corollary}

This is sharp up to the $O(1)$ factor whenever ${\rm K} \ll 1$. For instance, given  $d, L_1, \dots, L_d \in \NN$ and $v_0, \dots, v_d \in \mathbb{Z}$, we may set  $\A = \{ v_0 + l_1 v_1 + \dots + l_d v_d : 0 \leq l_i < L_i \}$. 
 Such sets are called \emph{generalised arithmetic progressions} of \emph{dimension} $d$ and satisfy $1 \leq {\rm K} =\exp(O(d))$, whence, Corollary \ref{dmf2} implies that $|\A|^5 \ll_d T(\A) \ll_d |\A|^5$. A special case of this is when $\A = [-N,N] \cap \ZZ$. In forthcoming work with Chapman  \cite{CM2025}, we prove an asymptotic formula for $T(\A)$ in the latter setting.

We now turn to the setting when the set $\A$ is multiplicatively structured. Many important results concerning the aforementioned sum-product phenomenon rely on studying additive equations over sets with small multiplicative doubling, see, for example, breakthrough work of Bourgain--Chang \cite{BC2004} on $k$-fold sumset-product set estimates. Employing the circle of ideas surrounding recent progress on this subject, we prove the following.

\begin{theorem} \label{multbd}
  There exists an absolute constant $C>0$ such that for any ${\rm M} >1$ and any finite, non-empty set $\A \subset \RR$ with $|\A \cdot \A| = {\rm M}|\A|$, one has 
    \[ {\rm M}^{-6}|\A|^5 \ll T(\A)  \ll {\rm M}^{C} |\A|^5 . \]
\end{theorem}

This is sharp up to the $O({\rm M}^C)$ factor. A notable aspect of our upper bound above is that we are able to circumvent $O(|\A|^{o(1)})$ factor losses which frequently appear in this type of setting, see for instance \cite[Proposition 2]{BC2004} and \cite[Theorem 1.3]{PZ2021}. A similar comment holds for our upper bounds in Corollary \ref{dmf2} and \eqref{nim}.

As in the setting of the Erd\H{o}s--Szemer\'{e}di sum-product conjecture, it is often much harder to obtain almost-optimal sum-product estimates in the case when $\A$ is multiplicatively structured in comparison to when $\A$ is additively structured. Moreover, while the methods involved in Theorem \ref{ptpl} and Corollary \ref{dmf2} are incidence geometric in nature, Theorem \ref{multbd} exploits a variety of ideas from additive combinatorics and number theory, including some of the recently developed techniques related to the polynomial Freiman--Ruzsa conjecture, see \cite{GGMT2023, PZ2021}. A key component in the proof of Theorem \ref{multbd} involves proving that sets of real numbers with small multiplicative doubling exhibit few solutions to the equation $x_1 + x_2 = x_3 + x_4$, and we prove the following result in this direction.

\begin{Proposition} \label{diag}
    Let $\A \subset \RR$ be a finite, non-empty set with $|\A \cdot \A| = {\rm M}|\A|$ for some ${\rm M} >1$, and let $\nu: \RR \to [0,1]$ satisfy $\supp(\nu) \subseteq \A$ and $\sum_{a \in \A}\nu(a) \leq 1$.  Then
    \begin{equation}  \label{glbb}
    \sum_{a_1, a_2, a_3, a_4 \in \A} \nu(a_1)\dots\nu(a_4) \mathds{1}_{a_1 + a_2 = a_3 + a_4} \ll ( {\rm M}  \cdot \log (4/\n{\nu}_2))^{O(1)} \n{\nu}_2^4.  
    \end{equation}
    The logarithmic factor $\log (4/\n{\nu}_2)$ on the right hand side can be removed if $\nu = |\A'|^{-1} \mathds{1}_{\A'}$ for some non-empty set $\A' \subseteq \A$.
\end{Proposition}

While such results, also often referred to as belonging to the ``few products, many sums" phenomenon, are somewhat new for the case when $\A \subset \RR$, they have been well studied in the integer setting. The first result of this flavour was proven by Chang \cite{Ch2003} over $\ZZ$, which was then quantitatively strengthened in the aforementioned work of Bourgain--Chang \cite{BC2004}. This was improved and simplified in subsequent work of P\'{a}lv\"{o}lgyi--Zhelezov \cite{PZ2021}, and generalised to non-linear equations in \cite{Mu2024b}. 

The proof methodologies from \cite{Ch2003, BC2004, PZ2021, Mu2024b} rely crucially on factorisation properties over integers, and so, they do not generalise to the case when $\A \subset \RR$. In this setting, previous work of Chang \cite{Ch2009}, as well as some of our own recent work \cite{Mu2024a}, involved proving related sum-product type estimates conditional on the weak Polynomial Freiman--Ruzsa conjecture. Building on these ideas, our proof of Proposition \ref{diag} combines the resolution of the latter conjecture by Gowers--Green--Manners--Tao \cite{GGMT2023} with a quantitative version of Schmidt's subspace theorem \cite{ESS2002}. We refer the reader to \cite{Ch2009, GGMT2023, GMT2023, Mu2024a} and the references therein for further details about the weak Polynomial Freiman--Ruzsa conjecture over the integers.

Proposition \ref{diag} may be combined in a natural fashion with Balog--Szemer\'{e}di--Gowers theorem \cite{Gow1998,Sch2015} to prove that any set $\A \subset \RR$ can be decomposed as $\A = \mathcal{B} \cup \mathcal{C}$ such that 
\[ E(\mathcal{B}) = |\{(b_1, \dots, b_4) \in \mathcal{B}^4 : b_1 - b_2 = b_3  -b_4\}|
\]
is almost as small as it can be, while $M(\mathcal{C})$ is slightly smaller than the trivial upper bound. Indeed, given $\delta >0$, if $M(\mathcal{A}) \leq |\mathcal{A}|^{3 - \delta}$, then we may set $\mathcal{C} = \mathcal{A}$ and we would be done, otherwise, we may apply Balog--Szemer\'{e}di--Gowers theorem to deduce that there exists $\mathcal{A}_1 \subseteq \mathcal{A}$ such that $|\mathcal{A}_1| \gg |\mathcal{A}|^{1-\delta}$ and $|\mathcal{A}_1 \cdot \mathcal{A}_1| \ll |\mathcal{A}_1|^{1 + O(\delta)}$. We can now apply Proposition \ref{diag} to deduce that $E(\mathcal{A}_1) \ll_{\delta} |\mathcal{A}_1|^{2 + O(\delta)}$. Applying the same procedure with $\mathcal{A} \setminus \mathcal{A}_1$, we can iteratively obtain a decomposition $\mathcal{A} = \mathcal{A}_1 \cup \dots \mathcal{A}_r \cup \mathcal{C}$ such that $r \ll_{\delta} |\mathcal{A}|^{ \delta}$ and $E(\mathcal{A}_i) \ll_{\delta} |\mathcal{A}_i|^{2 + O(\delta)}$ for every $1 \leq i \leq r$ and $M(\mathcal{C}) \ll |\mathcal{C}|^{3 - \delta}$, see \cite[Lemma $4.2$]{Mu2021}. We may now set $\mathcal{B} = \mathcal{A}_1 \cup \dots \cup \mathcal{A}_r$ and apply Lemma \ref{wtun} to obtain the claimed estimate. Such results are known as low-energy decompositions, see \cite{Mu2021, Mu2024b}. Since we have a non-trivial upper bound for $M(\mathcal{C})$, we may use \eqref{nim} to deduce that $T(\mathcal{C}) = o(|\mathcal{C}|^{5+1/2})$, while as $E(\mathcal{B})$ is almost extremally small, we may employ \eqref{spct} to deduce that $T(\mathcal{B})$ is extremally small. We record this as follows.

\begin{Corollary} \label{dve}
There exists an absolute constant $C>0$ such that for any $\delta>0$ and any finite, non-empty set $\mathcal{A} \subset \RR$, one can find disjoint sets $\mathcal{B}, \mathcal{C}$ with $\mathcal{A} = \mathcal{B} \cup \mathcal{C}$ such that 
\[  T(\mathcal{B}) \ll |\mathcal{B}|^3 E(\mathcal{B}) \ll_{\delta} |\mathcal{B}|^{5 + C\delta} \  \ \text{and} \ \ T(\mathcal{C}) \ll |\mathcal{C}|^4 M(\mathcal{C})^{1/2} \ll_{\delta} |\mathcal{C}|^{5+1/2 - \delta} .\]    
\end{Corollary}

Noting the above result, one might be interested in considering upper bounds for $T(\A)$ which depend purely on $|\A|$. For instance, inequality \eqref{nim} implies that $T(\A) \ll |\A|^{5+1/2}$ while all of our lower bounds for $T(\A)$ are of the order $O(|\A|^5)$. It is natural to wonder which of these is sharp, and in order to provide a partial answer to this, we turn to a related problem. 
Thus, given $a \in \mathbb{R}\setminus\{0\}$ and $b \in \RR$, we define the affine transformation $g_{a,b}: \RR \to \RR$ as the map satisfying $g_{a,b}(x) = ax + b$ for all $x\in \RR$. Given matrices $A = (a_{i,j})_{1 \leq i,j \leq 2}$ and $B = (b_{i,j})_{1 \leq i,j \leq 2}$ with non-zero off-diagonal entries, we note that 
\begin{equation} \label{conngig}
AB = BA \ \ \ \ \text{if and only if} \ \ \ \
g_{a_{1,2}, a_{2,2}}^{-1} \circ g_{a_{2,1}, a_{1,1}} =  g_{b_{1,2}, b_{2,2}}^{-1} \circ g_{b_{2,1}, b_{1,1}} .
\end{equation}
Writing $\frak{G} = \{ g_{a,b} : a \in \A \setminus \{0\} \ \text{and} \ b \in \A\}$ and $E(\frak{G})$ to be the number of $g,g',h,h' \in \frak{G}$ such that $g^{-1} \circ g' = h^{-1} \circ h'$, we may amalgamate \eqref{conngig} and Lemma \ref{tn4} to deduce that
\begin{equation} \label{dpt5}
    T(\A) = E(\frak{G}) + O(|\A|^5). 
\end{equation}
Hence, estimating $T(\A)$ is intimately connected to the growth in groups phenomenon, and in particular, to bounds on energies of subsets of the affine group. Focusing on the latter, upper bounds for $E(\frak{G})$ akin to \eqref{nim} with extra logarithmic factors on the right hand side were proven by Rudnev--Shkredov \cite{RS2022}.
Moreover, combining, in an ingenious manner, inverse results from  the higher energy method of Shkredov \cite{Shk2013} as well as some key input from the celebrated growth in groups phenomenon (see \cite{Hel2008,Mur2021}), Rudnev--Shkredov \cite{RS2022} showed that $E(\frak{G}) \ll |\A|^{5 + 1/2 - c}$ for some absolute constant $c>0$. With this in hand, \eqref{dpt5}  gives us
\begin{equation}  \label{rs18}
T(\A) \ll |\A|^{5 + 1/2 - c}.
\end{equation}

It is natural to ask whether such estimates extend to the setting of Theorem \ref{ptpl}, that is, whether $T(\mu_{\nu}) \ll \n{\nu}_2^{5 + c}$ holds whenever $0 \notin \supp(\nu)$, with $c>0$ being some absolute constant. While this does not seem to follow from \eqref{rs18} due to a lack of triangle inequality of the form \eqref{approxtriangle}, upon further exploiting \eqref{conngig}, we are able to prove the desired bound.


\begin{theorem} \label{rsk24}
  There exists some $c>0$ such that for any finitely--supported probability measure $\nu$ on $\RR$, one has 
    \[ T(\mu_{\nu}) \ll \n{\nu}_2^{5 + c}  + \nu(0)^3 .  \]
\end{theorem}

Theorem \ref{rsk24} extends \eqref{rs18} to arbitrary finitely-supported probability measures. Our proof of Theorem \ref{rsk24} utilises results from the above-mentioned work of Rudnev--Shkredov \cite{RS2022} and further exploits the relationship between $T(\A)$ and energies over affine groups. In particular, given a suitable decomposition of $\nu$ as $\nu = \sum_{i=1}^r \nu_i$, where $r \ll \log(2/\n{\nu}_2)$ and 
 each $\nu_i: \RR \to [0,1]$ is either almost-flat or very small in $l^{\infty}$ norm, it would be ideal to have a triangle inequality type bound of the shape 
\begin{equation}  \label{approxtriangle}
T(\mu_{\nu}) \ll r^{O(1)} \max_{1 \leq i \leq r} T(\mu_{\nu_i}). 
\end{equation}
While we are unable to prove such an inequality, we can bound $T(\mu_{\nu})$ by weighted multiplicative energies of a specific family of affine transformations over $\RR$, which in turn allows us to prove an approximate version of \eqref{approxtriangle}, see \eqref{hu1} and the remark at the end of \S7. Applying \eqref{hu1}, we are now required to bound pairs of commuting matrices with off-diagonal entries and diagonal entries arising from different subsets of $\RR$. We estimate these by means of \cite[Lemma $20$]{RS2022} as well as modified versions of Theorem \ref{ptpl}, see \eqref{asym1}.

Finding optimal upper bounds for $E(\frak{G})$ is an important question in the topic of growth in groups and has applications to many combinatorial geometric problems and sum-product type estimates, see \cite{PR2022} for further details and references. 
Moreover, while Theorem \ref{rsk24} employs ideas from growth in groups to estimate $T(\mu_{\nu})$, it would be interesting to incorporate, via \eqref{conngig}, the techniques involved in the study of commutating matrices into the subject of growth in groups, with the aim of obtaining stronger upper bounds for $E(\frak{G})$.  Considering Corollary \ref{dmf2}, Theorem \ref{multbd} and \eqref{rs18}, it is natural to speculate the following.

\begin{Conjecture} \label{nhtl}
For all $\varepsilon>0$ and all non-empty, finite sets $\A \subset \RR$, one has 
   \[ T(\A) \ll_{\varepsilon} |\A|^{5 + \varepsilon}. \]
\end{Conjecture}

Noting \eqref{dpt5}, this is equivalent to proving that $E(\frak{G}) \ll_{\varepsilon} |\frak{G}|^{5/2 + \varepsilon}$ for all $\varepsilon >0$ whenever $|\A| \geq 2$. A more general conjecture about energies of sets of the form $\frak{G}$ is presented in \cite{PR2022}.

Revisiting the setting of Theorem \ref{t1}, we note that there are many  examples of finitely-supported probability measures on $\M_2(\RR)$ such that $T(\mu) \gg \delta(\mu).$ In order to see this, choose some integer $N \geq 2$ and define
\begin{equation} \label{examplesharp}
    \mathscr{A} = \bigg\{ \begin{pmatrix}
    n & 2^j\\
    2^k & n
\end{pmatrix} :  n,j,k \in [N]  \bigg\} \ \ \text{and} \ \ \mu = N^{-3} \mathds{1}_{\mathscr{A}}. 
\end{equation} 
One can deduce that $\delta(\mu) \leq 1/N,$ say, via a straightforward application of the Schwarz--Zippel lemma. Moreover, given any matrices $X, Y \in \mathscr{A}$ such that
\[   X =   \begin{pmatrix}
    n & 2^j\\
    2^k & n
\end{pmatrix}  \ \ \text{and}  \ \ Y = \begin{pmatrix}
    n' & 2^{j'}\\
    2^{k'} & n'
\end{pmatrix}  \]
for some $n,n',j,j',k,k' \in [N]$, we have $XY = YX$ if and only if $j-j' = k - k'$. Thus,
\begin{equation}  \label{jrlk}
T(\mu) = N^{-6} \sum_{n,n',j,j',k,k' \in [N]} \mathds{1}_{j-j' = k-k'} =  \bigg( \frac{2}{3} - o(1) \bigg) \frac{1}{N} \geq   \bigg( \frac{2}{3} - o(1) \bigg)\delta(\mu).  
\end{equation}

We will now provide a brief outline of our paper. We employ \S2 to present some preliminary results from additive combinatorics, number theory and discrete geometry that we will require for our proofs of Theorems \ref{ptpl}, \ref{multbd} and \ref{rsk24}. We present our proof of Theorem \ref{t1} in \S3. In \S4, we employ incidence geometric techniques to prove Theorem \ref{ptpl}. We apply similar ideas in \S5  to prove Corollary \ref{dmf2}. We dedicate \S6 to utilising various results from additive combinatorics and number theory to prove Proposition \ref{diag} and then employing it to derive Theorem \ref{multbd}.
 In \S7, we elaborate on the connection between $T(\mu_{\nu})$ and energies of affine transformations over $\RR$, and we then present our proof of Theorem \ref{rsk24}.

\textbf{Notation.} We will use Vinogradov notation, that is, we write $Y \ll_{z} X$, or equivalently $Y =O_z(X)$, to mean that $|Y| \leq C_z X$ where $C_z>0$ is some constant depending on the parameter $z$. 
For any set $X$ and any $k \in \mathbb{N}$, we write $X^k = \{(x_1, \dots, x_k) : x_1, \dots, x_k \in X\}$. For every $d \in \mathbb{N}$, we use boldface to denote vectors $\vec{v} = (v_1, \dots, v_d) \in \RR^d$.  Given $\lambda \in \RR \setminus \{0\}$, a matrix $(a_{i,j})_{1 \leq i, j \leq 2} \in \M_2(\RR)$ and a set $\A \subset \RR$, we let $\lambda (a_{i,j})_{1 \leq i, j \leq 2} = (\lambda a_{i,j})_{1 \leq i, j \leq 2}$ and $\lambda \cdot \A = \{\lambda a : a \in \A\}$. 

\textbf{Acknowledgements.} We would like to thank Thomas Bloom, Tim Browning, Jonathan Chapman, Sam Chow, Sean Eberhard, Yifan Jing and Misha Rudnev for helpful comments.


\section{Preliminaries}

We begin by introducing some notation. Hence, given finite, non-empty sets $\mathcal{A}, \mathcal{B} \subset \RR$, we define the sumset $\A + \mathcal{B}$ and the product set $\mathcal{A} \cdot \mathcal{B}$  as
\[  \A + \mathcal{B} = \{ a+ b: a \in \A, b \in \mathcal{B}\} \ \text{and} \   \A \cdot \mathcal{B} = \{ a\cdot b: a \in \A, b \in \mathcal{B}\}. \]
Similarly, we denote the difference set 
\[ \mathcal{A} - \mathcal{B} =  \{ a - b: a \in \A, b \in \mathcal{B}\}. \]
Moreover, when $0 \notin \mathcal{B}$, we define the quotient set 
\[ \mathcal{A} / \mathcal{B} = \{ a/b : a \in \A, b \in \mathcal{B}\}. \]
Next, we recall that when $\A \subset \RR$ is a finite, non-empty set, then 
\[ E(\A) = \sum_{a_1, \dots, a_4 \in \A} \mathds{1}_{a_1 + a_2 = a_3 + a_4} \ \ \text{and} \ \ M(\A) = \sum_{a_1, \dots, a_4 \in \A \setminus \{0\}} \mathds{1}_{a_1/a_2 = a_3/ a_4}.  \]
When $\A$ is the empty set $\emptyset$, we set $|\A| = T(\A) = E(\A) = M(\A) = 0$. Given $p \geq 1$ and some finite, non-empty set $X$ and some function $f: X \to [0, \infty)$, we define
\[ \n{f}_p^p = \sum_{x \in X} f(x)^p \ \ \text{and} \ \ \n{f}_{\infty} = \max_{x \in X} f(x).  \]
Moreover, we also write $\supp(f)$ to denote the set $\{ x \in X : f(x) \neq 0\}$. For every finite, non-empty set $X \subset \M_2(\RR)$, we define ${\rm span}(X)$ to be the additive vector space spanned by elements of $X$ over $\RR$.

In our proof of Theorem \ref{ptpl}, we will use the following weighted version of the Szemer\'{e}di--Trotter theorem. 

\begin{lemma} \label{wtst}
    Let $P \subset \RR^2$ be a finite, non-empty set, let $\L$ be a finite, non-empty set of lines in $\RR^2$, let $w_P: P \to (0,\infty]$ and $w_{\L}: \L \to (0,\infty]$ be functions. Then 
    \begin{align*}
        \sum_{p \in P, l \in L} w_P(p) w_{\L}(l) \mathds{1}_{p \in l} 
        \ll & \n{w_P}_2^{2/3} \n{w_P}_1^{1/3}\n{w_{\L}}_2^{2/3} \n{w_{\L}}_1^{1/3} \\
         & + \n{w_P}_1 \n{w_{\L}}_{\infty} + \n{w_{P}}_{\infty} \n{w_{\L}}_{1}. 
    \end{align*}
\end{lemma}

A version of this where $w_P(P)$ and $w_{\L}(\L)$ are finite subsets of $\mathbb{N}$ was proved in \cite[Theorem 48]{Lu2017}, see also \cite[Lemma 3.1]{Mu2023}. The more general case can be derived from this by appropriately dilating the weight functions $w_P$ and $w_{\L}$ and then approximating them with the closest power of $2$.



We now record the well-known weak polynomial Freiman-Ruzsa conjecture, which was very recently proven to be true by Gowers--Green--Manners--Tao, see \cite[Theorem 1.3]{GGMT2023}. 

\begin{lemma} \label{wkpf}
Let $V= \mathbb{Z}^D$ for some $D\geq 1$, let $\A \subseteq V$ be a finite set with $|\A+\A| \leq K|\A|$ for some $K \geq 1$. Then there exists $\A_1 \subseteq \A$ with $|\A_1| > |\A|/K^C$ and elements $\xi_1, \dots, \xi_d \in V$ for some $d < C \log 2K$ such that
\[ \A_1 \subset \mathbb{Z} \xi_1 + \dots + \mathbb{Z} \xi_d, \]
where $C>0$ is some absolute constant. 
\end{lemma}




This will be used in combination with the following quantitative refinement of the subspace theorem of Evertse, Schmidt and Schlikewei \cite{ESS2002} as proved by Amoroso--Viada \cite{AV2009}. 

\begin{lemma} \label{subs}
Let $l, r \in \mathbb{N}$, let $\Gamma$ be a subgroup of $\mathbb{C}^{\times}$ of finite rank $r$, let $c_1, \dots, c_{l} \in \mathbb{C}^{\times}$. Then the number of solutions of the equation
\[ c_1 z_1 + \dots + c_l z_l = 1, \]
with $z_i \in \Gamma$ and no subsum on the left hand side vanishing is at most $(8l)^{4l^4(l + lr + 1)}$.
\end{lemma}

Here, we denote the multiplicative subgroup $\Gamma \subseteq \mathbb{C}^{\times}$ to have rank $r$ if there exists a finitely generated subgroup $\Gamma_0$ of $\Gamma$, again of rank $r$, such that the factor group $\Gamma/\Gamma_0$ is a torsion group.

Next, we will require Ruzsa's covering lemma, see \cite[Lemma 2.14]{TV2006}. We will state this below alongside the well-known Pl\"{u}nnecke--Ruzsa inequality \cite[Corollary 6.29]{TV2006}, see \cite{Pe2012} for a nice proof of the latter.

\begin{lemma} \label{rzcov}
    Let $G$ be an abelian group, let $\A, \mathcal{B} \subset G$ be finite, non-empty sets such that $|\A + \mathcal{B}| \leq K|\A|$ for some $K \geq 1$. Then there exists a non-empty set $X \subset G$ such that 
    \[ |X| \leq K \ \ \text{and} \ \ \mathcal{B} \subseteq X + \A - \A. \]
Moreover, for any $m,n \in \mathbb{N}$, one has
    \begin{equation} \label{prineq}
      |m \mathcal{B}| \leq K^m |\A| \ \ \text{and} \ \   |m\mathcal{B} - n\mathcal{B}| \leq K^{m+n}|\A|.
    \end{equation} 
\end{lemma}



We now record an equivalent formulation of a nice result of Rudnev--Shkredov \cite[Lemma $20$]{RS2022} on upper bounds for energies of affine transformations, which will play an important role in the proof of Theorem \ref{rsk24}. 

\begin{lemma} \label{rudshk}
Given $\delta>0$, there exists some constant $c \gg_{\delta} 1$ such that for any finite, non-empty sets $C \subset \RR \setminus \{0\}$ and $D \subset \RR$ with $|D|^{\delta} \leq |C| \leq |D|^{2}$, one has
\[  \sum_{c_1, \dots, c_4 \in C}\sum_{d_1, \dots, d_4 \in D} \mathds{1}_{c_1/c_3 = c_2/c_4} \mathds{1}_{c_4(d_1 - d_2) = c_2 (d_3 - d_4)} \ll_{\delta}  |C|^{5/2 - c}|D|^3.  \]
\end{lemma}

Given a group $G$ and some finitely-supported function $\nu : G \to [0,1]$ and a finite, non-empty set $A \subseteq G$, we define
\begin{equation} \label{rhn4}
    E_{\nu}(A) = \sum_{a_1, \dots, a_4 \in A} \nu(a_1) \dots \nu(a_4) \mathds{1}_{a_1 \cdot a_2^{-1} = a_3 \cdot a_4^{-1}}. 
\end{equation}
We conclude this section by proving the following elementary lemma.


\begin{lemma}  \label{wtun}
    Let $G$ be a group with group operation $\cdot$ and let $\nu: G \to [0,1]$ be a finitely-supported function. For every finite, non-empty sets $\A_1, \dots, \A_4 \subseteq G$, one has
    \begin{equation} \label{fiun}
        \sum_{a_1 \in \A_1, \dots, a_4 \in \A_4} \nu(a_1) \dots \nu(a_4) \mathds{1}_{a_1 \cdot a_2^{-1} = a_3 \cdot a_4^{-1}}  \leq \prod_{i=1}^4 E_{\nu}(\A_i)^{1/4}.  
    \end{equation} 
 Moreover, given $s\geq 1$ and finite, non-empty sets $\A_1, \dots, \A_{s} \subseteq G$ , we have
    \begin{equation} \label{union3}
         E_{\nu}(\A_1 \cup \dots \cup \A_s)^{1/4} \leq \sum_{i=1}^s E_{\nu}(\A_i)^{1/4}. 
    \end{equation}
\end{lemma} 

\begin{proof} The proof of Lemma \ref{wtun} is somewhat standard and elementary, see for instance, \cite[Lemma 3.4]{Mu2021} for the case when $G = \mathbb{R}$ and $\nu(x) = 1$ for all $x\in G$. We begin by deducing \eqref{union3} from \eqref{fiun} in a straightforward manner. Indeed, we have
\begin{align*}
    E_{\nu}(\A_1 \cup \dots \cup \A_s) 
    & \leq \sum_{1 \leq i_1, \dots, i_4 \leq s} \ \sum_{a_{i_1} \in \A_{i_1}}\nu(a_{i_1}) \dots \nu(a_{i_4}) \mathds{1}_{a_1 \cdot a_2^{-1} = a_3 \cdot a_4^{-1}}   \\
    & \leq \sum_{1 \leq i_1, \dots, i_4 \leq s} E_{\nu}(\A_{i_1})^{1/4}\dots E_{\nu}(\A_{i_s})^{1/4 } = (\sum_{1 \leq i \leq s} E_{\nu}(\A_i)^{1/4} )^4.
\end{align*}
Thus it suffices to prove \eqref{fiun}. In this endeavour, define $r_{i,j}(x) = \sum_{a \in \A_i, b \in \A_j} \nu(a) \nu(b) \mathds{1}_{x = a \cdot b^{-1}}$ for every $1 \leq i,j \leq 4$. We now apply the Cauchy--Schwarz inequality to deduce that
\[         \sum_{a_1 \in \A_1, \dots, a_4 \in \A_4} \nu(a_1) \dots \nu(a_4) \mathds{1}_{a_1 \cdot a_2^{-1} = a_3 \cdot a_4^{-1}}   = \sum_{x \in G} r_{1,2}(x) r_{3,4}(x) \leq (\sum_{x\in G} r_{1,2}(x)^2 )^{1/2} (\sum_{x\in G} r_{3,4}(x)^2 )^{1/2}.\]
Therefore it suffices to show that $\sum_{x \in G} r_{1,2}(x)^2 \leq E_{\nu}(\A_1)^{1/2}E_{\nu}(\A_2)^{1/2}$. Applying double--counting and Cauchy--Schwarz inequality as before, we get that
\begin{align*}
    \sum_{x \in G} r_{1,2}(x)^2 
    & = \sum_{\substack{a_1, a_3 \in \A_1,  \\ a_2, a_4 \in \A_2}} \nu(a_1) \dots \nu(a_4) \mathds{1}_{a_1 \cdot a_2^{-1} = a_3 \cdot a_4^{-1}}  = \sum_{\substack{a_1, a_3 \in \A_1,  \\ a_2, a_4 \in \A_2}} \nu(a_1) \dots \nu(a_4) \mathds{1}_{a_3^{-1}\cdot a_1 = a_4^{-1} \cdot a_2} \\
    &  \leq \prod_{i \in \{1,2\}} \Big( \sum_{a_1, \dots, a_4 \in \A_i} \nu(a_1) \dots \nu(a_4) \mathds{1}_{a_1^{-1}\cdot a_2 = a_3^{-1}\cdot a_4} \Big)^{1/2} \\
    & = \prod_{i \in \{1,2\}} \Big( \sum_{a_1, \dots, a_4 \in \A_i} \nu(a_1) \dots \nu(a_4) \mathds{1}_{ a_2 \cdot a_4^{-1} = a_1 \cdot a_3^{-1}} \Big)^{1/2}  = E_{\nu}(\A_1)^{1/2} E_{\nu}(\A_2)^{1/2},
\end{align*}
and so, we finish the proof of \eqref{fiun}.
\end{proof}





\section{Proof of Theorem \ref{t1}}

Writing $I$ to be the $2 \times 2$ identity matrix and $\mathcal{I} = \{ \lambda I : \lambda \in \RR\}$, we note that the contribution to $T(\mu)$ when either $X \in \mathcal{I}$ or $Y \in \mathcal{I}$ is at most 
\begin{equation}  \label{trivcontribution}
2\mu(\mathcal{I}) \sum_{Z \in {\rm supp}(\mu)} \mu(Z) = 2\mu(\mathcal{I}) \leq 2 \delta(\mu)  
\end{equation}
whence it suffices to estimate the contribution from the cases when $X, Y \in {\rm supp}(\mu) \setminus \mathcal{I}$. 
Writing $\mathcal{T} = {\rm supp}(\mu) \setminus \mathcal{I}$, let $H_1$ be the set of all $(Y_1, Y_2, Y_3) \in \mathcal{T}^3$ such that $Y_1, Y_2$ and $Y_3$ are linearly independent. Moreover, let $H_2 = \mathcal{T}^3 \setminus H_1$.

\begin{lemma} \label{s3dep}
    Let $(Y_1, Y_2, Y_3) \in H_1$ and let $X \in \M_2(\RR)$ such that $XY_i = Y_i X$ for every $1 \leq i \leq 3$. Then $X \in \mathcal{I}$. 
\end{lemma}

\begin{proof}
    
 The hypothesis implies that $XY = YX$ holds for every $Y$ satisfying  $Y = \lambda_1 Y_1 + \lambda_2 Y_2 + \lambda_3 Y_3$ for some $\lambda_1, \lambda_2, \lambda_3 \in \RR$. Let 
\[ V = {\rm span}(\{Y_1, Y_2, Y_3\} ) \ \ \text{and} \ \ U = \{ (v_{i,j})_{1 \leq i,j \leq 2} : v_{1,1} = v_{2,1} = 0\} .\]
Since $V$ is a $3$-dimensional subspace of $\M_2(\RR)$ and $U$ is a $2$-dimensional subspace of $\M_2(\RR)$, the set $V \cap U$ must contain at least a $1$-dimensional subspace. Let $Z \in (V \cap U) \setminus \{0\}$ and let 
\[ X =  \begin{pmatrix}
    x_1 & x_2\\
    x_3 & x_4
\end{pmatrix} \ \ \text{and let} \ \ Z  = \begin{pmatrix}
    0 & a\\
    0 & b
\end{pmatrix} \] 
for some $a, b \in \RR$ with $(a,b) \neq (0,0)$. Since $Z \in V$, we have that $XZ = ZX$, which in turn gives us
\[ ax_3 = 0 \ \ \text{and} \ \ (x_1 - x_4) a = -b x_2 \ \ \text{and} \ \ 0 =  - b x_3 .  \]
If $x_3 \neq 0$, we would get that $a = b =0$, which contradicts the fact that $(a,b) \neq (0,0)$. Thus, we must have $x_3 = 0$. 

Similarly, setting $U' = \{ (v_{i,j})_{1 \leq i,j \leq 2} : v_{1,1} = v_{1,2} = 0\}$, we see that $V \cap U'$ must contain at least a $1$-dimensional subspace. Let $Z' \in (V \cap U) \setminus \{0\}$ and let 
\[ Z' =  \begin{pmatrix}
    0 & 0\\
    e & f
\end{pmatrix} \]
for some $e,f \in \RR$ with $(e,f) \neq 0$. As before, we have that $XZ' = Z' X$, which in turn gives us
\[  0 = -f x_2 \ \ \text{and} \ \ x_2 e = 0 \ \ \text{and} \ \ -f x_3 = (x_1 - x_4) e. \]
Noting the fact that $(e, f) \neq (0,0)$, we get that $x_2 = 0$.

Finally, writing $U'' = \{ (v_{i,j})_{1 \leq i,j \leq 2} : v_{2,1} = v_{1,2} = 0\}$, we see that $V \cap U'$ must contain at least a $1$-dimensional subspace. Thus, let $Z'' \in U'' \cap V$ such that
\[ Z'' =  \begin{pmatrix}
    0 & h\\
    g & 0
\end{pmatrix}\]
for some $g,h \in \RR$ with $(g,h) \neq 0$. Since $X Z'' = Z'' X$, we get that
\[ x_1 h = x_4 h \ \ \text{and} \ \ g x_4 = gx_1.\]
Since $(g,h) \neq (0,0)$, we must have $x_1 = x_4$, which combines with the fact that $x_2,x_3 = 0$ to imply that $X \in \mathcal{I}$.
\end{proof}

We now return to the proof of Theorem \ref{t1}. We start by applying H\"{o}lder's inequality to get that
\begin{align}  \label{hldr}
\Big(\sum_{X, Y \in \mathcal{T}} \mu(X) \mu(Y) \mathds{1}_{XY = YX} \Big)^3
&  \leq \Big(\sum_{X \in \mathcal{T}} \mu(X) \Big)^2 \Big( \sum_{X \in \mathcal{T}} \mu(X) \Big( \sum_{Y \in \mathcal{T}} \mu(Y) \mathds{1}_{XY = YX} \Big)^3 \Big) \nonumber \\
\leq \sum_{X \in \mathcal{T}} \mu(X) \sum_{Y_1, Y_2, Y_3 \in \mathcal{T}} & \mu(Y_1) \mu(Y_2) \mu(Y_3) \ \mathds{1}_{XY_1 = Y_1X} \mathds{1}_{XY_2 = Y_2X} \mathds{1}_{XY_3 = Y_3X}.
\end{align}
We refer to the right hand side in \eqref{hldr} as $\Sigma$. Since $\mathcal{T} \cap \mathcal{I} = \emptyset$, Lemma \ref{s3dep} implies that the contribution of the cases when $(Y_1, Y_2, Y_3) \in H_1$ is zero. Thus, we may assume that $(Y_1, Y_2, Y_3) \in H_2$, that is, there exists some permutation $\sigma$ of $\{1,2,3\}$ such that
\[ Y_{\sigma(1)} \in {\rm span}(\{ Y_{\sigma(2)}, Y_{\sigma(3)}\}). \]
In this case, we see that $Y_{\sigma(1)}$ must lie in some subset $\mathcal{T}' \subseteq \mathcal{T}$ such that $\dim(\mathcal{T}') \leq 2$. Henceforth, we have
\begin{align} \label{cas1}
  \Sigma & \leq  \sum_{(Y_1, Y_2, Y_3)  \in H_2}  \sum_{X \in \mathcal{T}}  \mu(X) \mu(Y_1) \mu(Y_2) \mu(Y_3) \ \mathds{1}_{XY_1 = Y_1X} \mathds{1}_{XY_2 = Y_2X} \mathds{1}_{XY_3 = Y_3X} \nonumber \\
     & \ \ \ \  \leq 6\sum_{X \in \mathcal{T}}  \mu(X) \sum_{Y_{\sigma(2)}, Y_{\sigma(3)} \in \mathcal{T}}  \mu(Y_{\sigma(2)}) \mu(Y_{\sigma(3)})  \mathds{1}_{XY_{\sigma(2)} = Y_{\sigma(2)}X} \mathds{1}_{XY_{\sigma(3)} = Y_{\sigma(3)}X}  \max_{\substack{ \mathcal{T'} \subseteq {\rm supp}(\mu) , \\  \dim(\mathcal{T}') \leq 2
  }} \mu(\mathcal{T}') \nonumber \\
  & \ \ \ \ \leq 6 \ \delta(\mu) \sum_{X \in \mathcal{T}} \mu(X) \sum_{Y,Y' \in \mathcal{T}} \mu(Y) \mu(Y') \mathds{1}_{XY=YX} \mathds{1}_{XY' = Y'X}  \nonumber \\
  & \ \ \ \ = 6 \ \delta(\mu) \sum_{X \in \mathcal{T}} \mu(X) \bigg( \sum_{Y \in \mathcal{T}} \mu(Y) \mathds{1}_{XY = YX}\bigg)^2 \nonumber \\
  &  \ \ \ \ \leq 6 \ \delta(\mu) \bigg( \sum_{X \in \mathcal{T}} \mu(X) \bigg( \sum_{Y \in \mathcal{T}} \mu(Y) \mathds{1}_{XY = YX}\bigg)^3 \bigg)^{2/3}  \bigg(\sum_{X \in \mathcal{T} }  \mu(X) \bigg)^{1/3} \leq 6 \ \delta(\mu) \Sigma^{2/3},
\end{align}
where the first inequality in \eqref{cas1} follows from H\"{o}lder's inequality. Simplifying, we get that
\[ \Sigma \leq 6^3 \delta(\mu)^3.\]
Substituting this into \eqref{hldr}, we discern that
\[ \sum_{X, Y \in \mathcal{T}} \mu(X) \mu(Y) \mathds{1}_{XY = YX} \leq  \Sigma^{1/3} \leq 6 \ \delta(\mu).  \]
Putting this together with \eqref{trivcontribution} finishes the proof of Theorem \ref{t1}.



\section{Proof of Theorem \ref{ptpl}}

Given some finitely-supported probability measure $\nu$ on $\RR$, let
\[ X = \begin{pmatrix}
    x_1 & x_2\\
    x_3 & x_4
\end{pmatrix}  \ \ \text{and} \ \ Y = \begin{pmatrix}
    y_1 & y_2\\
    y_3 & y_4
\end{pmatrix}  \]
for some $x_1, \dots, y_4 \in \supp(\nu)$. Then $XY = YX$ holds if and only if
\begin{equation} \label{xy}
    x_2 y_3 = x_3 y_2 \ \ \text{and} \ \ x_2(y_4 - y_1) = y_2(x_4 - x_1) \ \ \text{and}  \ \ x_3(y_4 - y_1) = y_3(x_4 - x_1). 
\end{equation} 
Considering the solutions to \eqref{xy} satisfying $x_2 = y_2$ and $x_3 = y_3$, we obtain the lower bound
\begin{equation} \label{gkid}
    T(\mu_{\nu}) \geq \n{\nu}_2^4 \sum_{x_1, x_4, y_1, y_4 \in \supp(\nu)} \nu(x_1) \nu(x_4)\nu(y_1)\nu(y_4) \mathds{1}_{y_4 - y_1 = x_4 - x_1} = \n{\nu}_2^4 E_{\nu}(\supp(\nu)).
\end{equation} 

Let 
\[ H= \{ (x_1, \dots, y_4) \in \supp(\nu)^8 \ : \ x_1, \dots, y_4 \ \ \text{satisfy} \ \ \eqref{xy}\}, \] 
and 
\[ H' = \{(x_1, \dots, y_4) \in \supp(\nu)^8 \ : x_2x_3 y_2 y_3  (x_4 - x_1)(y_4 - y_1) = 0 \}. \]
Moreover, let $H_1 = H \cap H'$ and let $H_2 = H \setminus H'$. With this notation in hand, we get that
\begin{align} \label{sop}
   T(\mu_{\nu}) & = \sum_{X, Y \in {\rm supp}(\mu_{\nu})} \mu_{\nu}(X) \mu_{\nu}(Y) \mathds{1}_{XY = YX} \nonumber \\
   & = \sum_{(x_1, \dots, y_4) \in H_1}  \nu(x_1) \dots \nu(y_4) + \sum_{(x_1, \dots, y_4) \in H_2}\nu(x_1) \dots \nu(y_4). 
\end{align} 

We will now estimate the first sum on the right hand side of \eqref{sop}. 

\begin{lemma} \label{tn4}
    We have that
\[     \sum_{(x_1, \dots, y_4) \in H_1} \nu(x_1)\dots\nu(y_4)  \ll \nu(0)^3 + \n{\nu}_2^6. \]
\end{lemma}

\begin{proof}
Suppose $(x_1, \dots, y_4) \in H_1$ with $x_2 = 0$. Then we must have either $x_3 = 0$ or $y_2 = 0$. If $(x_2,x_3) = (0,0)$, then we must either have $y_3 = 0$ or $x_4 = x_1$. Similarly, if $(x_2,y_2) = (0,0)$, then we have to count solutions to $x_3 (y_4 - y_1) = y_3 (x_4 - x_1)$. Here, either $x_3(y_4 - y_1) \neq 0$ or $(x_3, y_3) = (0,0)$ or $(x_3, y_1 - y_4) = (0,0)$ or $(y_4 - y_1, y_3) = 0$ or $(y_4 - y_1, x_4 - x_1) = (0,0)$. Summarising this, we get that
\begin{align*}
\sum_{(x_1, \dots, y_4) \in H_1, x_2 = 0} \nu(x_1) \dots \nu(y_4)  \ll & \nu(0)^3 +  \nu(0)^2 \sum_{x \in \supp(\nu)} \nu(x)^2  +  \nu(0)^4 \\
+ \nu(0)^3 \sum_{x \in \supp(\nu)} & \nu(x)^2 
 + \nu(0)^2 (\sum_{x \in \supp(\nu)} \nu(x)^2)^2 + \nu(0)^2 T_1(\nu) 
\end{align*}
where 
\[ T_1(\nu) = \sum_{\substack{x_1,  x_4, y_1, y_4 \in \supp(\nu),\\ x_3, y_3 \in \supp(\nu) \setminus \{0\}}} \nu(x_1) \dots \nu(y_4) \mathds{1}_{x_3 (y_4 - y_1) = y_3 (x_4 - x_1)} . \]
Since $\nu$ is a probability measure, we have $\n{\nu}_2^2 \leq 1$ and $\n{\nu}_{\infty} \leq 1$. This allows us to simplify the preceding inequality to deduce that
\begin{equation} \label{starless1}
     \sum_{(x_1, \dots, y_4) \in H_1, x_2 = 0} \nu(x_1) \dots \nu(y_4)  \ll \nu(0)^3 + \nu(0)^2 (\sum_{x \in \supp(\nu)}\nu(x)^2 + T_1(\nu)). 
\end{equation} 

We can estimate $T_1(\nu)$ via a straightforward application of the Cauchy-Schwarz inequality
\begin{align*}
    T_1(\nu)
    & \leq \sum_{\substack{y_1, y_4 \in \supp(\nu),\\ x_3, y_3 \in \supp(\nu) \setminus \{0\}}} \nu(x_3) \nu(y_1)\nu(y_3)\nu(y_4) \sum_{x_1 \in \supp(\nu)} \nu(x_1) \nu\Big(x_1 + \frac{x_3}{y_3}(y_4 - y_1) \Big) \\
& \leq  \sum_{\substack{y_1, y_4 \in \supp(\nu),\\ x_3, y_3 \in \supp(\nu) \setminus \{0\}}} \nu(x_3) \nu(y_1)\nu(y_3)\nu(y_4)  \sum_{x_1 \in \supp(\nu)} \nu(x_1)^2   \leq \n{\nu}_2^2.
\end{align*}
 Substituting this into \eqref{starless1} and performing a similar analysis when $x_3 = 0$ or $y_2 = 0$ or $y_3 = 0$, we find that
\begin{align} \label{starless2}
    \sum_{\substack{(x_1, \dots, y_4) \in H_1, \\ x_2x_3y_2y_3 = 0}} \nu(x_1) \dots \nu(y_4) \ll \nu(0)^3 + \nu(0)^2 \n{\nu}_2^2  .
\end{align}

We now consider the contribution from the cases when $(x_1, \dots, y_4) \in H_1$ but $x_2x_3y_2y_3 \neq 0$. This would imply that $x_1 = x_4$, which in turn may be substituted into \eqref{xy} to get that $y_1 = y_4$. Thus we have
\begin{align*}
\sum_{\substack{(x_1, \dots, y_4) \in H_1, \\ x_2x_3y_2y_3 \neq 0}} &   \nu(x_1) \dots \nu(y_4) 
 \leq \sum_{\substack{x_1, y_1 \in \supp(\nu), \\ x_2, x_3, y_2, y_3 \in \supp(\nu)\setminus\{0\}  }} \nu(x_1)^2\nu(y_1)^2 \nu(x_2)\dots\nu(y_3) \mathds{1}_{x_2 y_3 = x_3 y_2} \\
& = \sum_{x_1, y_1 \in \supp(\nu)} \nu(x_1)^2\nu(y_1)^2 \sum_{x_2, x_3, y_2 \in \supp(\nu)\setminus\{0\}  }\nu(x_2)\nu(x_3) \nu(y_2) \nu(y_2 x_3/x_2) \\
& \leq (\sum_{x \in \supp(\nu)} \nu(x)^2 )^2  \sum_{x_2, y_2 \in \supp(\nu)} \nu(x_2) \nu(y_2) \sum_{x_3 \in \supp(\nu)} \nu(x_3)^2 = \n{\nu}_2^6,
\end{align*}
where the last inequality follows from the Cauchy-Schwarz inequality. Combining this with \eqref{starless2}, we deduce that
\[     \sum_{(x_1, \dots, y_4) \in H_1} \nu(x_1)\dots\nu(y_4) \ll \nu(0)^3 + \nu(0)^2 \n{\nu}_2^2 + \n{\nu}_2^6 \ll \nu(0)^3 + \n{\nu}_2^6, \]
which is the desired upper bound.
\end{proof}

We will now estimate the second sum on the right hand side of \eqref{sop}. 

\begin{lemma} \label{ender2}
    We have
    \[   \sum_{(x_1, \dots, y_4) \in H_2} \nu(x_1) \dots \nu(y_4) \ll \n{\nu}_2^{4}  M(\nu)^{1/2}  +  \n{\nu}_{\infty}^2 M(\nu) + \n{\nu}_2^6 .\]
\end{lemma}

Note that this, when combined with \eqref{sop} and Lemma \ref{tn4}, immediately gives us 
\[T(\mu_{\nu}) \ll \n{\nu}_2^{4}  M(\nu)^{1/2}  +  \n{\nu}_{\infty}^2 M(\nu) + \n{\nu}_2^6 + \nu(0)^3, \]
thus proving Theorem \ref{ptpl}. Thus, our main aim for this section is to prove Lemma \ref{ender2}.

\begin{proof}[Proof of Lemma \ref{ender2}]
Given $y, z \in \RR$, we define
\[ q(z) = \sum_{x_2, y_2 \in \supp(\nu)\setminus \{0\} }  \nu(x_2) \nu(y_2) \mathds{1}_{z = x_2/y_2} \ \ \text{and} \ \ r_z(y) = \sum_{a_1, a_2 \in \supp(\nu)}\nu(a_1)\nu(a_2) \mathds{1}_{y = a_1 + za_2}.\]
A standard double counting implies that
\begin{align} \label{fmbc}
     \sum_{(x_1, \dots, y_4) \in H_2} \nu(x_1) \dots \nu(y_4) 
     & \leq 
 \sum_{(x_1, \dots, y_4) \in H_2 } \nu(x_1) \dots \nu(y_4) \mathds{1}_{\frac{x_2}{y_2} = \frac{x_3}{y_3} = \frac{x_4 - x_1}{y_4 - y_1}}  \nonumber \\
 & \leq \sum_{z \in \supp(q)} q(z)^2 \sum_{y \in \supp(r_z)} r_z(y)^2  \nonumber \\
 & = \sum_{\substack{z \in \supp(q), \\ y \in \supp(r_z)}} q(z)^2 r_z(y) \sum_{a_1, a_2 \in \supp(\nu) } \nu(a_1) \nu(a_2) \mathds{1}_{y = a_1 + a_2 z}.
\end{align}
We are now ready to apply Lemma \ref{wtst}, and so, we record some properties of the functions $q$ and $r_z$. For instance, note that whenever $z \in \supp(q) \subset \RR \setminus \{0\}$ and $y \in \supp(r_z)$, one has
\begin{equation} \label{pcw}
    q(z) \leq \n{\nu}_2^2 \ \ \text{and} \ \ r_z(y) \leq  \n{\nu}_2^2 \ \ \text{and} \ \ \sum_{x \in \supp(q)} q(x) \leq 1 \ \ \text{and}   \ \ \sum_{w \in \supp(r_z)} r_z(w) \leq 1, 
\end{equation} 
where the first two upper bounds follow from the Cauchy--Schwarz inequality. Moreover, we have that
\[ \n{q}_2^2 =  \sum_{z \in \supp(q)} q(z)^2  = M(\nu).   \]

Putting together \eqref{fmbc} along with Lemma \ref{wtst}, we deduce that
\begin{align*}
\sum_{\substack{z \in \supp(q), \\ y \in \supp(r_z)}} q(z)^2  r_z(y)^2 
& \ll ( \sum_{\substack{z \in \supp(q), \\ y \in \supp(r_z)}} q(z)^4 r_z(y)^2 )^{1/3}  ( \sum_{\substack{z \in \supp(q), \\ y \in \supp(r_z)}} q(z)^2 r_z(y) )^{1/3} \n{\nu}_2^{4/3} \n{\nu}_1^{2/3} \\
& \ \ \ \ \ \  + \n{\nu}_{\infty}^2 ( \sum_{\substack{z \in \supp(q), \\ y \in \supp(r_z)}} q(z)^2 r_z(y) ) +  \max_{\substack{z \in \supp(q), \\ y \in \supp(r_z)}} q(z)^2 r_z(y) \n{\nu}_1^2 .
\end{align*}
Utilising the various bounds listed in \eqref{pcw}, one sees that
\begin{align*}
    \sum_{\substack{z \in \supp(q), \\ y \in \supp(r_z)}} q(z)^2  r_z(y)^2  
    & \ll \n{q}_{\infty}^{2/3} (  \sum_{\substack{z \in \supp(q), \\ y \in \supp(r_z)}} q(z)^2  r_z(y)^2)^{1/3} M(\nu)^{1/3}  \n{\nu}_2^{4/3} +  \n{\nu}_{\infty}^2 M(\nu) + \n{\nu}_2^6 \\
    & \ll \n{\nu}_2^{8/3} (  \sum_{\substack{z \in \supp(q), \\ y \in \supp(r_z)}} q(z)^2  r_z(y)^2)^{1/3} M(\nu)^{1/3}+  \n{\nu}_{\infty}^2 M(\nu) + \n{\nu}_2^6 .
\end{align*} 
Simplifying the above, we find that
\[  \sum_{\substack{z \in \supp(q), \\ y \in \supp(r_z)}} q(z)^2  r_z(y)^2  \ll \n{\nu}_2^{4}  M(\nu)^{1/2}  +  \n{\nu}_{\infty}^2 M(\nu) + \n{\nu}_2^6 .\]
Substituting this upper bound in \eqref{fmbc} furnishes the desired estimate. 
\end{proof}

It is worth noting that the our method can also treat asymmetric cases. For example, for all finite, non-empty sets $C, D\subset \RR$ with $0\notin C$, the above techniques imply that
\begin{equation} \label{asym1}
\sum_{\substack{c_1, \dots, c_4 \in C, \\ d_1, \dots, d_4 \in D}} \mathds{1}_{\frac{c_1}{c_2} = \frac{c_3}{c_4} = \frac{d_1 - d_2}{d_3 - d_4} }  \ll |C|  M(C)^{1/2}|D|^{3} + M(C) |D|^2  \ll |C|^{5/2}|D|^3 + |C|^3 |D|^2.
\end{equation}
The left hand side can be interpreted as counting commuting pairs of matrices whose diagonal and off-diagonal entries arise from sets $D$ and $C$ respectively.

\section{Proof of Corollary \ref{dmf2}}

Let $\mathcal{A} \subset \RR$ be a  finite, non-empty set. The lower bound $T(\A) \geq |\A|^5 {\rm K}^{-1}$ follows from combining \eqref{gkid} with the fact that $E(\A) \geq |\A|^3{\rm K}^{-1}$, the latter following from Cauchy--Schwarz inequality. Hence, we focus on the upper bound. Writing $H$ to be the set of all $(a_1, \dots, a_8) \in \A^8$ such that
\[ a_2 a_7 - a_3 a_6 = a_2(a_8 - a_5) - a_6(a_4 - a_1) = a_3( a_8 - a_5) - a_7(a_4 - a_1) = 0 \]
and setting
\[ H_1 = \{ (a_1, \dots, a_8) \in H : a_2a_3a_6a_7 (a_4 - a_1)(a_8 - a_5) = 0 \} \ \ \text{and}  \ \ H_2 = H \setminus H_1, \]
we see that $T(\A) = |H|$. Upon setting $\nu = |\A|^{-1} \mathds{1}_{\A}$, Lemma \ref{tn4} implies that
\[ |H_1| \ll |\A|^8(\nu(0)^3 + \n{\nu}_2^6) \ll |\A|^5 \leq |\A + \A|^{1 + 1/3} |\A|^{3 + 2/3}, \]
whence it suffices to study the contribution from $H_2$. Note that
\begin{equation}  \label{tauba}
|H_2| = \sum_{(a_1, \dots, a_8) \in H_2} \mathds{1}_{a_2/a_6 = a_3/a_7 = (a_4 - a_1)/(a_8 - a_5)} \leq \sum_{z \in \A_1/\A_1} q_1(z)^2 r_1(z), 
\end{equation}
where 
\[ \A_1 = \A \setminus \{0\}, \ \ \text{and} \ \ q_1(z) = \sum_{a, b \in \A_1} \mathds{1}_{z = a/b}, \ \ \text{and}  \ \ r_1(z) = \sum_{\substack{a, b, c, d \in \A, \\ a \neq b \ \text{and} \ c \neq d}} \mathds{1}_{z = (a-b)/(c-d)}. \]
Applying H\"{o}lder's inequality, we get that
\begin{equation} \label{4rt}
     |H_2| \leq (\sum_{z \in \A_1/\A_1} q_1(z)^3 )^{2/3}  (\sum_{z \in \A_1/\A_1} r_1(z)^3 )^{1/3}.  
\end{equation}

The rest of the proof relies on the following two key lemmas, the first of these arising from the aforementioned work of Solymosi \cite{So2009}.

\begin{lemma} \label{so}
    Given $\tau \geq 1$, let $Q_{\tau} = \{ z \in \RR  : \tau \leq q_1(z) < 2 \tau\}$. Then 
    \[ |Q_{\tau}| \ll \frac{|\A + \A|^2}{\tau^2} . \]
\end{lemma}
 
It is worth noting that Solymosi proved a slightly different result, that is, given finite, non-empty sets $\mathcal{C}, \mathcal{D} \subset (0, \infty)$, for every $\tau \geq 1$, the set $Q_{\mathcal{C}, \mathcal{D}, \tau}$ of all elements $z \in \RR$ with 
\[ \tau \leq \sum_{c \in \mathcal{C}, d \in \mathcal{D}} \mathds{1}_{z = c/d} < 8 \tau\]
satisfies 
\begin{equation} \label{dietpepsi} 
|Q_{\mathcal{C}, \mathcal{D}, \tau}|  \ll \frac{|\mathcal{C} + \mathcal{C}||\mathcal{D} + \mathcal{D}|}{\tau^2}. 
\end{equation}
In order to deduce Lemma \ref{so} from the above, write $\A_1 = \A' \cup \A''$, where $\A' \subset (0, \infty)$ and $\A'' \subset (-\infty,0)$, and note that if some $z \in \RR$ satisfies $2\tau > q_1(z) \geq \tau$, then $|z| \in Q_{\mathcal{C}, \mathcal{D}, \tau/4}$ for some choice of $\mathcal{C}, \mathcal{D} \in \{\A', -\A''\}$. For any such $\mathcal{C}, \mathcal{D}$, it is easy to see that $|\mathcal{C} + \mathcal{C}|,|\mathcal{D} + \mathcal{D}| \leq |\A+\A|$, and so, applying  \eqref{dietpepsi} delivers Lemma \ref{so}.

We now state the second lemma that we require. 

\begin{lemma} \label{jfh}
        Given $\tau \geq 1$, let $R_{\tau} = \{ z \in \RR : \tau \leq r_1(z) < 2 \tau\}$. Then we have
        \[ |R_{\tau}| \ll \frac{|\A|^6}{\tau^2} . \]
\end{lemma}

This was proved by Murphy--Roche-Newton--Shkredov \cite{MRNS2015}, see also \cite{RNR2015} for related estimates.  We provide a short proof of this here. 

\begin{proof}
 Given $z \in \supp(r_1) \subset \RR \setminus\{0\}$, we define, for every $y \in \RR$, the function 
 \[ s_z(y) = \sum_{a_1, a_2 \in \A_1} \mathds{1}_{y = a_1 + za_2}. \]
 Thus, we have that
\begin{equation} \label{pari}  
|R_{\tau}|\tau \leq  \sum_{z \in R_{\tau}} r_1(z)  = \sum_{z \in R_{\tau}} \sum_{\substack{a, b, c, d \in \A, \\ a \neq b \ \text{and} \ c \neq d}} \mathds{1}_{a + dz = b + cz}   \leq \sum_{z \in R_{\tau}} \sum_{y \in \supp(s_z)}  s_z(y) \sum_{b,c \in \A} \mathds{1}_{y = b + cz} . 
\end{equation}
We are now set to apply Lemma \ref{wtst} to estimate the sum on the right hand side. In particular, note that
\[ \sum_{z \in R_{\tau}} \sum_{y \in \supp(s_z)}  s_z(y) \leq |\A|^2 |R_{\tau}| \ \ \text{and} \ \ \sup_{z \in R_{\tau}, y \in \supp(s_z)} s_z(y)  \ll |\A|.  \]
Moreover, the second moment of the weight function $s_z$ is precisely the right hand side of \eqref{pari}, that is, 
\[  \sum_{z \in R_{\tau}} \sum_{y \in \supp(s_z)}  s_z(y)^2 =  \sum_{z \in R_{\tau}} \sum_{y \in \supp(s_z)}  s_z(y) \sum_{b,c \in \A} \mathds{1}_{y = b + cz}.  \]
Thus, upon applying Lemma \ref{wtst} and simplifying the ensuing inequality, one finds that
 \[       \sum_{z \in R_{\tau}} \sum_{y \in \supp(s_z)}   s_z(y) \sum_{b,c \in \A} \mathds{1}_{y = b + cz} \ll |\A|^3 |R_{\tau}|^{1/2}  + |\A|^2 |R_{\tau}|. \]
 Combining this with the preceding discussion, we get that whenever $\tau \geq C|\A|^2$ for some absolute constant $C>0$, one has $|R_{\tau}| \ll |\A|^6/\tau^2$. On the other hand, when $\tau < C|\A|^2$, we may use the fact that $\sum_{z \in R_{\tau}} r_1(z) = |\A|^4$ to discern that 
 \[ |R_{\tau}| \leq \frac{|\A|^4}{\tau} < \frac{C|A|^6}{\tau^2}. \qedhere \]
\end{proof}

We now return to estimate the right hand side of \eqref{4rt}, and so, letting $J_1$ be the smallest integer such that $|\A| \leq 2^{J_1}$, we apply Lemma \ref{so} to deduce that
\begin{align*}
    \sum_{z \in \A_1/\A_1} q_1(z)^3   \ll \sum_{j=0}^{J_1} |Q_{2^j}|2^{3j} \ll |\A + \A|^2 \sum_{j=0}^{J_1} 2^j \ll |\A + \A|^2 2^{J_1} \ll |\A + \A|^2 |\A|.
\end{align*} 
Moreover, letting $J_2$ be the smallest integer such that $|\A|^3 \leq 2^{J_2}$, we utilise Lemma \ref{jfh} to obtain the bound
\[  \sum_{z \in \A_1/\A_1} r_1(z)^3 \ll \sum_{j=0}^{J_2} |R_{2^j}|2^{3j} \ll |\A|^6 \sum_{j=0}^{J_2} 2^j \ll |\A|^6 2^{J_2} \ll |\A|^9.  \]
Substituting these estimates into \eqref{4rt} furnishes the desired result
\[    |H_2| \leq (\sum_{z \in \A_1/\A_1} q_1(z)^3 )^{2/3}  (\sum_{z \in \A_1/\A_1} r_1(z)^3 )^{1/3} \ll |\A + \A|^{4/3} |\A|^{3 + 2/3},  \]
and so, we conclude the proof of Corollary \ref{dmf2}.


\section{Proofs of Theorem \ref{multbd} and Proposition \ref{diag} }

Our aim in this section is to prove Theorem \ref{multbd} and Proposition \ref{diag}. For the purposes of this section, given any non-empty set $\mathcal{A}$ lying in either $(0, \infty)$ or $(-\infty,0)$, we denote ${\rm rk}(A)$ to be the rank of the multiplicative subgroup of $\mathbb{R}^{\times}$ generated by elements of $\mathcal{A}$. We will now present the following consequence of Lemma \ref{subs}. 

\begin{lemma} \label{sbap}
    Let $r \geq 1$ and let $\mathcal{B} \subset (0, \infty)$ be a finite, non-empty set such that ${\rm rk}(\mathcal{B}) = r$. Then for every $\emptyset \neq \mathcal{B}' \subseteq \mathcal{B}$, one has
    \begin{equation}  \label{cldiag}
    E(\mathcal{B}') \ll |\mathcal{B}'|^2 + \exp(O(r))|\mathcal{B}'|. 
    \end{equation}
\end{lemma}

\begin{proof}
   Let $\mathcal{B}'$ be a non-empty subset of $\mathcal{B}$, and suppose $b_1, \dots, b_4 \in \mathcal{B}'$ satisfy 
\begin{equation} \label{esty}
b_1 + b_2 - b_3 = b_4.
\end{equation}
Since $0 \notin \mathcal{B}'$ and $\mathcal{B}' \subset (0, \infty)$, the only way that there could exist some subsum on the left hand side of \eqref{esty} which sums to $0$  is if either $b_1 = b_3$ or $b_2 = b_3$. There are $O(|\mathcal{B}'|^2)$ possibilities for such a quadruple. Since this matches the upper bound for $E(\mathcal{B}')$ in \eqref{cldiag}, we may assume that there are no subsums summing to $0$ in the right hand side of \eqref{esty}. Fixing $b_4 \in \mathcal{B}'$, we see that any $b_1, b_2, b_3$ satisfying \eqref{esty} also satisfy
\begin{equation} \label{jfkg} 
\frac{b_1}{b_4} + \frac{b_2}{b_4} - \frac{b_3}{b_4} = 1. 
\end{equation}
Since $b_4^{-1} \cdot \mathcal{B}' \subseteq b_4^{-1} \cdot \mathcal{B} \subseteq \Gamma$ and the rank of $\Gamma$ is at most $O(r)$, Lemma \ref{subs} implies that there are at most $\exp(O( r))$ many solutions to \eqref{jfkg}. Putting this together with the preceding discussion furnishes the required estimate
\[ E(\mathcal{B}') \ll | \mathcal{B}'|^2 + \exp(O(r)) |\mathcal{B}'|. \qedhere\]
\end{proof}

We will now prove the aforementioned special case of Proposition \ref{diag} when $\nu =|\A'|^{-1}\mathds{1}_{\A'}$ for some non-empty $\A' \subset A$.

\begin{lemma} \label{unwtdiag}
    Let $K>1$ and let $\A \subset \RR$ be a finite, non-empty set such that $|\A \cdot \A| \leq K|\A|$. Then for any non-empty set $\A' \subseteq \A$, we have 
\[ E(\A') \ll K^{O(1)} |\A'|^2. \]
\end{lemma}
\begin{proof}
We may assume that $|\A'| \geq 10,$ since otherwise, we can use the trivial bound $E(\A') \leq |\A'|^3 \ll |\A'|^2$. Next, if $\{0,1,-1\} \cap \A' \neq \emptyset$, then the the number of $a_1, \dots, a_4 \in \A'$ such that $a_1 + a_2 = a_3 + a_4$ and $a_i \in \{0,1,-1\}$ for some $1 \leq i \leq 4$ can be seen to be $O(|\A'|^2)$. Moreover, writing $\A_1' = \A' \setminus \{0,1,-1\}$ and $\A_1 = \A \setminus \{0,1,-1\}$, we see that 
\[ |\A_1 \cdot \A_1| \leq |\A \cdot \A| \leq K|\A| \leq 2K |\A_1|. \]
Thus, it suffices to prove our result for the sets $\A_1' \subseteq \A_1$, that is, we may assume that $ \{0,1,-1\} \cap \A = \emptyset$. 

We will now show that there exists some $\mathcal{B} \subseteq \A$ such that 
\begin{equation} \label{cldiag100}
    |\mathcal{B}| \gg |\A|/K^{O(1)} \ \ \text{and} \ \ {\rm rk}(\mathcal{B} \cdot \mathcal{B}^{-1} ) = {\rm rk}(\mathcal{B}) \ll \log (2K).
\end{equation} 
 In order to see this, note that either $|\A \cap (0, \infty)| \geq |\A|/2$ or $|\A \cap (-\infty,0)| \geq |\A|/2$. Since $|\A\cdot \A|$ and the conclusion \eqref{cldiag100} are somewhat invariant under dilation of $\mathcal{B}$ and $\mathcal{A}$, we may assume that the set $\A_2 = \A \cap (0,\infty)$ satisfies $|\A_2| \geq |\A|/2$. Moreover, we have that
\[ |\A_2 \cdot \A_2| \leq |\A \cdot \A| \leq K |\A| \leq 2K |\A_2|.\]
Now, since $A_2 \subset (0, \infty) \setminus \{1\}$, it generates a multiplicative subgroup $V$ of $\mathbb{R}^{\times}$ which is isomorphic to $\mathbb{Z}^r$ for some $r \in \mathbb{N}$. We may now apply Lemma \ref{wkpf} to find some $\mathcal{B} \subseteq \A_2$ such that
\[ |\mathcal{B}| \gg |\A_2|/K^{O(1)} \gg |\A|/ K^{O(1)} \]
and the multiplicative group $\Gamma$ generated by $\mathcal{B}$ has rank at most $O(\log 2K)$. 


With Lemma \ref{sbap} and \eqref{cldiag100} in hand, we note that
 \[   |\A' \cdot \mathcal{B}| \leq |\A \cdot \A| \leq K |\A| \ll K^{O(1)} |\mathcal{B}|.  \]
Applying Lemma \ref{rzcov}, we find some set $\mathcal{S} = \{s_1, \dots, s_k\}  \subset \mathbb{R} \setminus \{0\}$ such that $\mathcal{A}' \subseteq \mathcal{S} \cdot \mathcal{B} \cdot \mathcal{B}^{-1}$ and $k \ll K^{O(1)}$. Define the sets $\mathcal{B}_1, \dots, \mathcal{B}_k$ iteratively by setting 
\begin{equation} \label{iterativedecom} 
\mathcal{B}_1 = \A' \cap (s_1 \cdot \mathcal{B} \cdot \mathcal{B}^{-1})\ \ \text{and} \ \  \mathcal{B}_i = ( \A' \cap (s_i \cdot \mathcal{B} \cdot \mathcal{B}^{-1})) \setminus ( \cup_{j=1}^{i-1}  \mathcal{B}_{j} ). 
\end{equation}
Thus, $\mathcal{B}_1, \dots, \mathcal{B}_k$ are pairwise disjoint sets such that $\A' = \cup_{i=1}^k \mathcal{B}_i$. For every $1 \leq i \leq k$, since $s_i^{-1} \cdot \mathcal{B}_i \subseteq \mathcal{B}  \cdot \mathcal{B}^{-1}$, Lemma \ref{sbap} and inequalities \eqref{cldiag} and \eqref{cldiag100} combine to furnish the bound
\[ E(\mathcal{B}_i) = E(s_i^{-1} \cdot \mathcal{B}_i) \ll  |s_i^{-1} \cdot \mathcal{B}_i|^2 + K^{O(1)} |s_i^{-1} \cdot \mathcal{B}_i| =|\mathcal{B}_i|^2 + K^{O(1)} | \mathcal{B}_i| . \]
Putting this together with a straightforward application of \eqref{union3}, we get that
\begin{align*}
     E(\mathcal{A}') & = E(\cup_{i=1}^k \mathcal{B}_i) \ll k^4 \max_{1 \leq i \leq k} E(\mathcal{B}_i)   \ll K^{O(1)} \max_{1 \leq i \leq k} | \mathcal{B}_i|^2 \ll K^{O(1)} |\A'|^2. 
\end{align*}
This concludes the proof of Lemma \ref{unwtdiag}.
\end{proof}

We are now ready to prove Proposition \ref{diag}.

\begin{proof}[Proof of Proposition \ref{diag}]
For ease of exposition, we use the notation in \eqref{rhn4} with $G = \RR$ and the group operation being addition, and define
\[ E_{\nu}(\mathcal{B}) = \sum_{b_1, \dots, b_4 \in \mathcal{B}} \nu(b_1) \dots \nu(b_4) \mathds{1}_{b_1 + b_2 = b_3 + b_4} \]
for every finite, non-empty set $\mathcal{B}\subset \RR$. Now, let $J>0$ be the largest integer such that $2^{-J} < \n{\nu}_2^8$, and for every $1 \leq i \leq J$, let $\A_i = \{ x \in \A :  2^{-i} \n{\nu}_2 < \nu(x) \leq 2^{-i+1} \n{\nu}_2 \}$. Set $\A_0 = \cup_{1 \leq i \leq J} \A_i$ and $\A' = \A \setminus\A_0$. Applying \eqref{union3}, we see that
   \[ E_{\nu}(\A) = E_{\nu}(\A_0 \cup \A') \ll E_{\nu}(\A') + E_{\nu}(\A_0) . \]
   Note that 
   \begin{align*}
        E_{\nu}(A') & =  \sum_{a_1, a_2, a_3 \in \A'} \nu(a_1) \nu(a_2) \nu(a_3) \nu(a_4) \mathds{1}_{\A'}(a_1 + a_2 - a_3) \\
        & \ll \n{\nu}_2^8  \sum_{a_1, a_2, a_3 \in \A'} \nu(a_1) \nu(a_2) \nu(a_3) \ll \n{\nu}_2^8 \nu(\A)^3 \ll \n{\nu}_2^8,
   \end{align*}
   and so, it suffices to estimate $E_{\nu}(\A_0)$. We begin by applying \eqref{union3} again to deduce that
   \[ E_{\nu}(\A_0)^{1/4} \leq \sum_{i=1}^J E_{\nu}(\A_i)^{1/4} \ll \sum_{i=1}^J 2^{-i} \n{\nu}_2 E(\A_i)^{1/4}.\]
   We may now use Lemma \ref{unwtdiag} to get that
   \[E_{\nu}(\A_0)^{1/4} \ll  {\rm M}^{O(1)} \sum_{i=1}^J 2^{-i}\n{\nu}_2 |\A_i|^{1/2}  \ll {\rm M}^{O(1)} \n{\nu}_2 J. \]
We finish our proof by noting that $2^J \ll 1/\n{\nu}_2^8$, whence, $J \ll \log (4/ \n{\nu}_2)$.
\end{proof}

We conclude this section by proving Theorem \ref{multbd}.

\begin{proof}[Proof of Theorem \ref{multbd}] 
Let $\A$ be a finite, non-empty set $\A \subset \RR$ such that $|\A \cdot \A| = {\rm M}|\A|$. We may assume that $|\A| \geq 10$ since otherwise we are done. We will first prove the upper bound in Theorem \ref{multbd}. Setting $\nu = |\A|^{-1} \mathds{1}_{\A}$, Lemma \ref{unwtdiag} implies that
\begin{equation} \label{blel}
    E(\A) \ll {\rm M}^{C} |\A|^2
\end{equation}
for some absolute constant $C>0$. Applying Lemma \ref{tn4} and writing $\A' = \A \setminus\{0\}$, we see that
\[ T(\A) \ll |\A|^5 + \sum_{z \in \A'/\A'} q(z)^2 \sum_{a_1, \dots, a_4 \in \A} \mathds{1}_{a_1 - a_2 = z(a_3 - a_4)}, \]
where $q(z) = \sum_{b_1, b_2 \in \A'} \mathds{1}_{b_1 = b_2 z}.$ Since $q(z) \leq |\A|$ for all $z \in \A'/\A'$ and $\sum_{z \in \A'/\A'}q(z) \leq |\A|^2$, one gets that
\[ T(\A) \ll |\A|^5 + |\A|^3\max_{z \in \A'/\A'} \sum_{a_1, \dots, a_4 \in \A} \mathds{1}_{a_1 - a_2 = z(a_3 - a_4)}.\]
As $0 \notin \A'/\A'$, Lemma \ref{wtun} gives us
\begin{equation} \label{spct}
T(\A) \ll |\A|^5 + |\A|^3 E(\A)^{1/2}  \max_{z \in \A'/\A'} E(z\cdot \A)^{1/2}  = |\A|^5 + |\A|^3 E(\A) .
\end{equation}
Putting this together with \eqref{blel} dispenses the desired upper bound. 

We now consider the lower bound. If $0 \in \A$, then we can consider solutions to \eqref{xy} satisfying $x_2 = x_3 = 0$ and $x_4 = x_1$ to get that $T(\A) \geq |\A|^5$, whereupon, it suffices to consider the case when $0 \notin \A$. Here, note that for any given $z \in \A/\A$, all $x_1, \dots, y_4 \in \A$ satisfying $x_i = z y_i$ for every $1 \leq i \leq 4$ also satisfy \eqref{xy}. In particular, writing $q'(z) = \sum_{a_1, a_2 \in \A}\mathds{1}_{z = a_2/a_1}$, we have that
\begin{equation} \label{multlbd}
    T(\A) \geq \sum_{z \in \A/\A} q'(z)^4 \geq \frac{(\sum_{z \in \A/\A} q'(z))^4}{|\A/\A|^3}, 
\end{equation} 
where the second step follows from H\"{o}lder's inequality. As before, we see that
\[ \sum_{z \in \A/\A} q'(z) = |\A|^2\]
while a standard application of the Pl\"{u}nnecke--Ruzsa inequality \eqref{prineq} furnishes the estimate
\[ |\A/\A| \leq (|\A\cdot \A|/|\A|)^2 |\A| = {\rm M}^2 |\A|. \]
Combining these with \eqref{multlbd} delivers the claimed lower bound.
\end{proof}

\section{Proof of Theorem \ref{rsk24}}

We begin this section by restating the connection between $T(\A)$ and energies over affine groups. For every $a \in \mathbb{R}\setminus \{0\}$ and $b \in \mathbb{R}$, we define 
\[ \frak{m}_{a,b} = \begin{pmatrix} a & b \\ 0 & 1
\end{pmatrix} .\]
Given a non-empty set $\A \subseteq \RR$, define the set
\[ {\rm Aff}(\A) = \{ \frak{m}_{a,b} : a \in \A \setminus \{0\}, b \in \A\} .\]
We note that ${\rm Aff}(\RR)$ is a group with the group operation being matrix multiplication. Given some finitely-supported function $\nu : \RR \to [0,1]$, we can define a finitely-supported function $\sigma_{\nu}: {\rm Aff}(\RR) \to [0,1]$ by setting $\sigma_{\nu}(\frak{m}_{a,b}) = \nu(a) \nu(b)$ for every $a \in \RR \setminus \{0\}$ and $b \in \RR$. Furthermore, note that
\[ \frak{m}_{a,b} \frak{m}_{x,y} = \frak{m}_{ax, ay +b} \ \ \text{and} \ \ \frak{m}_{a,b}^{-1} = \frak{m}_{a^{-1},-a^{-1}b} \]
holds for all $a,b,x,y \in \RR$ with $ax\neq 0$. Hence, given $a_1, \dots, b_4 \in \RR$ with $a_1a_3b_1b_3 \neq 0$, we see that $\frak{m}_{a_1,a_2} \frak{m}_{a_3,a_4}^{-1} = \frak{m}_{b_1,b_2} \frak{m}_{b_3,b_4}^{-1} $ holds if and only if
\[ a_1/a_3 = b_1/b_3 \ \ \text{and} \ \ b_3( a_2 - b_2) = b_1 (a_4 - b_4) . \]
Defining
\[ X = \begin{pmatrix}
    b_2 & a_1 \\
    b_1 & a_2
\end{pmatrix}
\ \ \text{and} \ \ Y = \begin{pmatrix}
    b_4 & a_3 \\
    b_3 & a_4 
\end{pmatrix}
\]
and recalling \eqref{xy}, we see that $\frak{m}_{a_1,a_2} \frak{m}_{a_3,a_4}^{-1} = \frak{m}_{b_1,b_2} \frak{m}_{b_3,b_4}^{-1} $ holds if and only if $XY = YX$. Combining this observation with Lemma \ref{tn4} and \eqref{rhn4}, we discern that
\begin{equation} \label{affen}
    T(\mu_{\nu}) \ll E_{\sigma_{\nu}}({\rm Aff}(\supp(\nu))) + \nu(0)^3 + \n{\nu}_2^6. 
\end{equation} 

We will now present our proof of Theorem \ref{rsk24}.

\begin{proof}[Proof of Theorem \ref{rsk24}]
    Thus, let $\nu$ be a probability measure such that $\A:=\supp(\nu)$ is a finite, non-empty subset of $\RR$. For every $j \in \mathbb{N}$, let $\A_j = \{ x \in \A : 2^{-j} < \nu(x) \leq 2^{-j+1}\} $ and let $J \in \mathbb{N}$ be the smallest integer such that $2^{-J} < \n{\nu}_2^{16}$. Finally, let $\A' = \A \setminus \cup_{1 \leq j \leq J}\A_j$. Noting \eqref{affen}, it suffices to show that $E_{\sigma_{\nu}}({\rm Aff}(\A)) \ll \n{\nu}_2^{5+ c}$ for some fixed constant $c>0$. 

For every $1 \leq i, j \leq J$, we let
\[ S_{i,j} = \bigg\{ \begin{pmatrix} a & b \\ 0 & 1
\end{pmatrix} : a \in \A_i \setminus \{0\}, b \in \A_j \bigg\} \ \ \text{and} \ \  S'= {\rm Aff}(\A) \setminus (\cup_{1 \leq i,j \leq J} S_{i,j} ). \]
Applying Lemma \ref{wtun}, we get that
\begin{equation} \label{hu1}
    E_{\sigma_{\nu}}({\rm Aff}(\A))^{1/4} \ll \sum_{1 \leq i,j \leq J} E_{\sigma_{\nu}}(S_{i,j})^{1/4} + E_{\sigma_{\nu}}(S')^{1/4}.
\end{equation}
From the definition of $S'$, we get that $\sigma_{\nu}(s) \ll 2^{-J} \ll \n{\nu}_2^{16}$ for every $s \in S'$, whence,
\begin{equation} \label{ert5}
E_{\sigma_{\nu}}(S') = \sum_{s_2, s_3 ,s_4\in S'} \sigma_{\nu}(s_2) \sigma_{\nu}(s_3)  \sigma_{\nu}(s_4) \sigma_{\nu}(s_3 \cdot s_4^{-1}\cdot s_2)  \ll \max_{s \in S'} \sigma_{\nu}(s)  \ll \n{\nu}_2^{16}. 
\end{equation}

For any $1 \leq i,j \leq J$, we see that
\begin{align*} 
E_{\sigma_{\nu}}(S_{i,j}) 
& = \sum_{c_1,\dots,c_4 \in \A_i \setminus \{0\}}\sum_{d_1,\dots,d_4 \in \A_j} \nu(c_1) \dots \nu(d_4) \mathds{1}_{c_1/c_2 = c_3/c_4} \mathds{1}_{c_4(d_1 - d_3) = c_3(d_2 - d_4)} \\
& \ll 2^{-4i-4j} \sum_{c_1,\dots,c_4 \in \A_i \setminus \{0\}} \sum_{d_1,\dots,d_4 \in \A_j} \mathds{1}_{c_1/c_2 = c_3/c_4} \mathds{1}_{c_4(d_1 - d_3) = c_3(d_2 - d_4)}.
\end{align*}
Writing $\A_i' = \A_i \setminus \{0\}$, if $|\A_j|^{1/2} \leq |\A_i'|\leq |\A_j|^2$,  then we may apply Lemma \ref{rudshk} to get that
\begin{align*}
    E_{\sigma_{\nu}}(S_{i,j}) \ll 2^{-4i} 2^{-4j} |\A_i'|^{5/2 - c_1} |\A_j|^3 
    = (|\A_i'| 2^{-2i})^{3/2 + c_1} (|\A_i'|2^{-i})^{1 - 2c_1}(|\A_j|2^{-2j})(|\A_j|2^{-j})^2
\end{align*}
holds for some absolute constant $c_1 \in (0,1)$. Since 
\begin{equation} \label{ndb}
2^{-2k} |\A_k| \ll \n{\nu}_2^2 \ \ \text{and} \ \ 2^{-k}|\A_k| \ll \n{\nu}_1 \leq 1 
\end{equation}
for every $k \in \{i,j\}$, we get that
\begin{equation} \label{hu2}
    E_{\sigma_{\nu}}(S_{i,j}) \ll \n{\nu}_2^{5 + 2c_1}.
\end{equation}
 
On the other hand, if $|\A_i'| < |\A_j|^{1/2}$, then \eqref{asym1} gives us
\begin{equation} \label{ans}
    E_{\sigma_{\nu}}(S_{i,j}) \ll 2^{-4i} |\A_i'|^{5/2} 2^{-4j}|\A_j|^3 + 2^{-4i}|\A_i'|^3 2^{-4j}|\A_j|^2.
\end{equation}
Utilising \eqref{ndb} immediately gives us that 
\[ 2^{-4i}|\A_i'|^3 2^{-4j}|\A_j|^2 \ll \n{\nu}_2^6,\]
whereupon, it suffices to estimate the first term on the right hand side in \eqref{ans}. We first consider the case when $2^{-2i} |\A_i'| \leq (2^{-2j}|\A_j|)^{15/16}$. Here, we have that
\begin{align*}
2^{-4i} |\A_i'|^{5/2} 2^{-4j}|\A_j|^3 
& \ll 2^{-15j/4- 4j}|\A_j|^{15/8+3} |\A_i'|^{1/2} \ll 2^{-7j - 3j/4} |\A_j|^{5 + 1/8} \\
& = (2^{-2j} |\A_j|)^{2 + 5/8} (2^{-j}|\A_j|)^{2+1/2} \ll \n{\nu}_2^{5 + 1/4}.
\end{align*}
Moreover, if $2^{-2j}|\A_j| < (2^{-2i} |\A_i'|)^{16/15}$, then 
\[ 2^{-4i} |\A_i'|^{5/2} 2^{-4j}|\A_j|^3 \ll 2^{-4i - 32i/15}|\A_i'|^{5/2 + 16/15} \ll 
\n{\nu}_2^{5 + 2/15}.
 \]
 In either case, whenever $|\A_i'| < |\A_j|^{1/2}$, one gets that
 \begin{equation} \label{hu3}
     E_{\sigma_{\nu}}(S_{i,j}) \ll \n{\nu}_2^{5 + 2/15}.
 \end{equation}

 We now turn to the case when $|\A_i'| \geq |\A_j|^{2}$. As in the preceding argument, we first consider the subcase when $(2^{-2i} |\A_i'|)^{7/8} \geq 2^{-2j}|\A_j|$. Here, we see that
 \[ 2^{-4i} |\A_i'|^{5/2} 2^{-4j}|\A_j|^3 \ll 2^{-4i - 7i/2}|\A_i'|^{5/2 + 7/4} |\A_j| \ll 2^{-7i-i/2}|\A_i'|^{4 + 3/4} \ll
 \n{\nu}_2^{5 + 1/2}.\]
 Moreover, if $2^{-2i}|\A_i'| < (2^{-2j}|\A_j|)^{8/7}$, then we have that
 \[  2^{-4i} |\A_i'|^{5/2} 2^{-4j}|\A_j|^3  \ll (2^{-2j}|\A_j|)^{\frac{3}{2} \cdot \frac{8}{7}} 2^{-2j} |\A_j| \ll \n{\nu}_2^{2(1 + 12/7)} = \n{\nu}_2^{5 + 3/7}. \]
 As in the previous case, putting these bounds together with \eqref{ans}, one gets that 
 \[ E_{\sigma_{\nu}}(S_{i,j}) \ll \n{\nu}_2^{5 + 1/2} . \]
 
 Substituting the above along with \eqref{ert5}, \eqref{hu2} and \eqref{hu3} into \eqref{hu1}, we find that
 \[ E_{\sigma_{\nu}}({\rm Aff}(\A)) \ll J^8 (\n{\nu}_2^{5 + 2c_1} +\n{\nu}_2^{5 + 2/15} +   \n{\nu}_2^{5 + 1/2} + \n{\nu}_2^{16} ),\]
 whereupon, noting the fact that $J \ll \log (2/\n{\nu}_2^{16}) \ll \log (2/\n{\nu}_2) \ll \n{\nu}_2^{-c_1/100}$ and setting $c = \min\{c_1/2,1/15\}$ gives us
 \[ E_{\sigma_{\nu}}({\rm Aff}(\A)) \ll \n{\nu}_2^{5 + c}.\]
 Amalgamating this with \eqref{affen} finishes the proof of Theorem \ref{rsk24}.
\end{proof}

In the above proof, we employ the connection between $T(\mu_{\nu})$ and growth in groups in two key ways, the first being our usage of Lemma \ref{rudshk} and the second being that \eqref{affen} and \eqref{hu1} allow us to upper bound $T(\mu_{\nu})$ by incidence-type estimates over two asymmetric sets. This can be seen as an approximate version of the following speculative triangle inequality which stipulates that for any finitely--supported functions $\nu, \nu_1, \dots, \nu_r: \mathbb{R} \to [0,1]$ satisfying $\nu = \nu_1 + \dots + \nu_r$, one has \eqref{approxtriangle}.
While we are unable to prove such a result, if true, this would provide an alternative proof of Theorem \ref{rsk24} that circumvents Lemma \ref{rudshk} and employs results such as Corollary \ref{dve}, thus giving better quantitative exponents.


\bibliographystyle{amsbracket}
\providecommand{\bysame}{\leavevmode\hbox to3em{\hrulefill}\thinspace}

\end{document}